\documentclass[11pt,leqno]{amsart}
\usepackage{amsmath,amssymb,amsthm,mathrsfs}
\usepackage{hyperref,tikz-cd,graphicx}
\usepackage{dot2texi}
\usepackage{hyperref}

\usepackage{color}

\usepackage{bbm}

\DeclareMathOperator{\co}{co}
\DeclareMathOperator{\dens}{dens}
\newcommand{\cconv}{\overline{\co}}

\renewcommand{\geq}{\geqslant}
\renewcommand{\leq}{\leqslant}

\newcommand{\norm}[1]{\left\Vert#1\right\Vert}

\newcommand{\spn}{\operatorname{span}}

\newcommand{\Xss}{X^{\ast\ast}}

\newcommand{\supp}{\operatorname{supp}}

\newcommand{\iten}{\widehat{\otimes}_\varepsilon}
\newcommand{\pten}{\widehat{\otimes}_\pi}

\newcommand{\e}{\varepsilon}
\newcommand{\cl}{\mathcal}
\newcommand{\n}{\norm}
\newcommand{\bb}{\mathbb}

\newcommand{\ASQ}[1]{\text{ASQ}_{<#1}} 
\newcommand{\SQ}[1]{\text{SQ}_{<#1}} 
\newcommand{\rsSQ}[3]{\left(#1,#2\right)\text{-SQ}_{<#3}}
\newcommand{\rsSQeq}[3]{\left(#1,#2\right)\text{-SQ}_{#3}}
\newcommand{\AandSQ}[1]{\text{(A)SQ}_{<#1}}

\newtheorem{theorem}{Theorem}[section]

\newtheorem{lemma}[theorem]{Lemma}

\newtheorem{proposition}[theorem]{Proposition}
\newtheorem{corollary}[theorem]{Corollary}
\theoremstyle{definition}
\newtheorem{definition}[theorem]{Definition}
\newtheorem{example}[theorem]{Example}

\theoremstyle{remark}
\newtheorem{remark}[theorem]{Remark}
\newtheorem*{question}{Question}
\numberwithin{equation}{section}

\def\fnote#1{\footnote}

\def\ignora#1{}
\def\n3#1{\left\vert  \! \left\vert \! \left\vert \, #1 \, \right\vert \!
  \right\vert \! \right\vert }

\def\nn#1{\left\vert  \! \left\vert \! \left\vert \, #1 \, \right\vert \!
	\right\vert \! \right\vert }

\begin{document}

\title{Transfinite almost square Banach spaces}

\author[Avil\'es]{Antonio Avil\'es}
\address[Avil\'es and Rueda Zoca]{Universidad de Murcia, Departamento de Matem\'{a}ticas, Campus de Espinardo 30100 Murcia, Spain} 
\email{\texttt{avileslo@um.es}}

\author {Stefano Ciaci}
\address[Ciaci, Langemets and Lissitsin]{Institute of Mathematics and Statistics, University of Tartu, Narva mnt 18, 51009 Tartu, Estonia}
\email{stefano.ciaci@ut.ee}

\author {Johann Langemets}
\email{johann.langemets@ut.ee}
\urladdr{\url{https://johannlangemets.wordpress.com/}}

\author {Aleksei Lissitsin}

\email{aleksei.lissitsin@ut.ee}

\author[Rueda Zoca]{Abraham Rueda Zoca}
\email{\texttt{abraham.rueda@um.es}}
\urladdr{\url{https://arzenglish.wordpress.com}}



\maketitle

\markboth{AVIL\'ES, CIACI, LANGEMETS, LISSITSIN AND RUEDA ZOCA}{TRANSFINITE ALMOST SQUARE BANACH SPACES}

\begin{abstract}
It is known that a Banach space contains an isomorphic copy of $c_0$ if, and only if, it can be equivalently renormed to be almost square. We introduce and study transfinite versions of almost square Banach spaces with the purpose to relate them to the containment of isomorphic copies of $c_0(\kappa)$, where $\kappa$ is some uncountable cardinal. We also provide several examples and stability results of the above properties by taking direct sums, tensor products and ultraproducts. By connecting the above properties with transfinite analogues of the strong diameter two property and octahedral norms, we obtain a solution to an open question from \cite{CLL2021}.

\end{abstract}

\section{Introduction}

Since the starting point of the study of Banach space theory, a considerable effort has been made in order to determine how the presence of an isomorphic copy of $c_0$ or $\ell_1$ in a Banach space affects its structure. This makes interesting the search of properties which characterise the containment of the abovementioned spaces. In this sense, let us point out two characterisations of the containment of the spaces $\ell_1$ and $c_0$ of geometric nature. In \cite[Theorem II.4]{Gode1989} it is proved that a Banach space $X$ contains an isomorphic copy of $\ell_1$ if, and only if, it admits an equivalent norm $\nn{\cdot}$ which is \textit{octahedral}, that is, given a finite-dimensional subspace $Y\subset X$ and $\varepsilon>0$, there is an element $x\in S_{(X,\nn{\cdot})}$ such that
\begin{equation*}
    \nn{ y+rx}\geq (2-\varepsilon)(\nn{y}+\vert r\vert)\quad \text{for all $y\in Y$ and $r\in\mathbb R$.}
\end{equation*}
Concerning the containment of $c_0$, a more recent characterisation was given in \cite[Corollary 2.4]{BLR2016}:  a Banach space $X$ contains an isomorphic copy of $c_0$ if, and only if, it admits an equivalent \textit{almost square} (\textit{ASQ}, for short) norm $\nn{\cdot}$, that is, given a finite-dimensional subspace $Y\subset X$ and $\varepsilon>0$, there is an element $x\in S_{(X,\nn{\cdot})}$ such that
\begin{equation*}
    \nn{ y+r x}\leq (1+\varepsilon)(\nn{ y}\vee\vert r\vert) \quad \text{for all $y\in Y$ and $r\in\mathbb R$.}
\end{equation*}

At this point, it is natural to look for geometric characterisations of the containment of non-separable versions of $\ell_1$ and $c_0$. In this spirit, as far as the containment of $\ell_1(\kappa)$ is concerned, transfinite generalisations of octahedral norms were introduced in \cite{CLL2021} in different directions and some characterisations of the containment of $\ell_1(\kappa)$ were obtained in \cite{AMRZ2021,CLL2021}. To mention the strongest known result, it is proved in \cite[Theorem 1.3]{AMRZ2021}  that a Banach space $X$ contains an isomorphic copy of $\ell_1(\kappa)$, where $\kappa$ is an uncountable cardinal, if, and only if, there exists an equivalent norm $\nn{\cdot}$ such that $(X,\nn\cdot)$ fails the \textit{$(-1)$-ball covering property for cardinals $<\kappa$} ($(-1)$-BCP$_{<\kappa}$, for short), that is satisfying that, given any subspace $Y\subset X$ such that $\dens(Y)<\kappa$, there exists $x\in S_{(X,\nn\cdot)}$ such that
\begin{equation*}
    \nn{ y+r x}=\nn{ y}+\vert r\vert \quad \text{for all $y\in Y$ and $r\in\mathbb R$.}
\end{equation*}

Motivated by the above results, in the present paper we aim to introduce different transfinite versions of ASQ spaces in order to search for a characterisation of those spaces that contain isomorphic copies of $c_0(\kappa)$.

Let us now describe in more detail the content of the paper. In Section \ref{section:defiandexamples} we define transfinite ASQ spaces and call them $\ASQ{\kappa}$ spaces, where $\kappa$ is some fixed cardinal (see Definition \ref{defi:transfiniteASQ}) and we provide many examples of Banach spaces enjoying these properties. In Section \ref{section:renormings} we consider the relations between being $\ASQ{\kappa}$, admitting an equivalent $\ASQ{\kappa}$ renorming and containing isomorphic copies of $c_0(\kappa)$. One of the highlights of this section is Example \ref{theo:kappaASQnotconc0}, in which we find, for every uncountable cardinal $\kappa$, a Banach space $X$ which is $\ASQ{\kappa}$, but such that $X^*$ fails to contain $\ell_1(\omega_1)$ and, in particular, $X$ cannot contain $c_0(\omega_1)$. This means that the property of being renormable to be $\ASQ{\kappa}$ is not strong enough to characterise the Banach spaces that contains $c_0(\kappa)$ isomorphically. Hence, we consider a strengthening of  $\ASQ{\kappa}$ that we call $\SQ{\kappa}$ spaces (see Definition \ref{defi:transfiniteASQ}) and which will contain isomorphic copies of $c_0(\kappa)$. Therefore we face the question whether every Banach space containing $c_0(\kappa)$, for some uncountable cardinal $\kappa$, admits an equivalent $\SQ{\kappa}$ renorming. Even though we do not know the answer in complete generality, we prove in Theorem \ref{theo:renormingkASQ} that, if $\dens(X)=\kappa$, then $X$ admits an equivalent $\SQ{\text{cf}(\kappa)}$ renorming, where $\text{cf}(\kappa)$ stands for the cofinality of $\kappa$.

In Section \ref{section:stabilityresults} we study various stability results for $\AandSQ{\kappa}$ spaces through different operations of Banach spaces, in order to enlarge the class of the known examples of Banach spaces enjoying these properties. We mainly extend known results for ASQ spaces, but which, in their transfinite version, produce new surprising results. For instance, with respect to absolute sums, in Theorem~\ref{thm: ell_infty sum} we are able to produce $\ell_\infty$-sums of spaces which are $\ASQ{\kappa}$ even though none of its components is $\ASQ{\kappa}$, which is a notable difference from the previously known results for finite sums of ASQ spaces. We also analyse $\AandSQ{\kappa}$ properties with respect to taking spaces of operators and tensor products. In Corollary~\ref{coro:injectivetensASQ} we prove that, if $X$ is $\AandSQ{\kappa}$ and $Y$ is non-trivial, then the injective tensor product $X\iten Y$ is $\AandSQ{\kappa}$. If we also require $Y$ being $\AandSQ{\kappa}$, then so is its projective tensor product $X\pten Y$ (see Theorem~\ref{theo:largeASQprojectensor}). Observe that the latter result is important, because most of the known examples of ASQ spaces come from some kind of $\infty$-norm, but the projective norm on a tensor product has dramatically different behaviour.

We end the study of stability results with ultraproducts which, as one might expect, provide a lot of examples of $\SQ{\kappa}$ spaces. Indeed, in Theorem~\ref{theo:stabilityultraproduct} we prove that, if a family $\{X_\alpha:\alpha\in\mathscr A\}$ consists of $\ASQ{\kappa}$ spaces and if we consider a countable incomplete ultrafilter $\mathcal U$, then the ultraproduct $(X_\alpha)_\mathcal U$ is $\SQ{\kappa}$. Lastly, we show in Example~\ref{exam:ultraln} that the requirement of the factors being $\ASQ{\kappa}$ is not necessary.

In Section \ref{section:connectotherprop} we investigate the connection of $\AandSQ{\kappa}$ properties with other properties of Banach space, such as the transfinite versions of octahedrality and diameter two properties. As a consequence of our work, we derive that if $X$ is $\ASQ{\kappa}$ (respectively, $\SQ{\kappa}$), then $X$ has the SD2P$_{<\kappa}$ (respectively, the $1$-ASD2P$_{<\kappa}$) and, consequently, $X^*$ is $<\kappa$-octahedral (respectively, fails the $(-1)$-BCP$_{<\kappa}$) (see Proposition \ref{prop:connecASQSd2P}). As a consequence, in Remark~\ref{rem:solutionCCL2021} we provide, for every uncountable cardinal $\kappa$, an example of a Banach space $X$ which is $<\kappa$-octahedral but which fails to contain an isomorphic copy of $\ell_1(\omega_1)$, giving a negative solution to \cite[Question 1]{CLL2021}. 

In Section~\ref{section:generalisations} we introduce a parametric version of $\AandSQ{\kappa}$ spaces (see Definition~\ref{def: parametric ASQ}) which includes all of the known versions of ASQ spaces. Moreover, we improve the known isomorphic characterization of Banach spaces containing $c_0$ (see Theorem~\ref{thm: (r,s)-SQ renorming}).

\textbf{Terminology.} Throughout the paper we only consider real Banach spaces. Given a Banach space $X$, we denote the closed unit ball and the unit sphere of $X$ by $B_X$ and $S_X$, respectively. We also denote by $X^*$ the topological dual of $X$. Given two Banach spaces $X$ and $Y$ we denote by $L(X,Y)$ the space of linear bounded operators from $X$ into $Y$. Given a subset $A$ of $X$ we denote by $\spn(A)$ (respectively, $\overline{\spn}(A)$) the linear span (respectively, the closed linear span) of the set $A$, whereas $\dens(X)$ denotes the density character of a topological space $X$, i.e. the smallest cardinal which is the cardinality of a dense set in $X$.


Given a set $A$, we denote by $\vert A\vert$ its cardinality and by $\mathcal{P}_\kappa(A)$ and $\cl P_{<\kappa} (A)$ the sets of all subsets of $A$ of cardinality at most $\kappa$ and strictly less than $\kappa$, respectively, for some cardinal $\kappa$. By $\mathbb N_{\geq 2}$ we denote the set $\{n\in\mathbb N\colon n\geq 2\}$.

Given an infinite set $\mathscr A$ and an uncountable cardinal $\kappa$, a non-principal ultrafilter $\mathcal U$ over $\mathscr A$ is said to be \textit{$\kappa$-complete} if it closed with respect to $<\kappa$ many intersections. It is immediate to see that a non-principal ultrafilter $\mathcal U$ is $\aleph_1$-incomplete if, and only if, there is a function $f:\mathscr A\longrightarrow \mathbb R$ so that $f(\alpha)>0$ for every $\alpha\in\mathscr A$ and so that $\lim_\mathcal U f(\alpha)=0$.

\section{Transfinite almost square Banach spaces and first examples}\label{section:defiandexamples}

Let us begin with the definition of an (A)SQ Banach space depending on a given cardinal.

\begin{definition}\label{defi:transfiniteASQ}
Let $X$ be a Banach space and $\kappa$ a cardinal. We say that
\begin{itemize}
 \item[(a)] $X$ is \textit{$<\kappa$-almost square} ($\ASQ{\kappa}$, for short) if, for every set $A\in\mathcal{P}_{<\kappa}(S_X)$ and $\varepsilon>0$, there exists $y\in S_X$ such that $\Vert x \pm y\Vert\leq 1+\varepsilon$ holds for all $x\in A$,
 
 \item[(b)] $X$ is \textit{$<\kappa$-square} (${SQ}_{<\kappa}$, for short) if, for every set $A\in\mathcal{P}_{<\kappa}(S_X)$, there exists $y\in S_X$ such that $\Vert x \pm y\Vert\leq 1$ holds for all $x\in A$.
\end{itemize}
As a special case, let us also define \emph{ASQ}$_\kappa$ and \emph{SQ}$_\kappa$ spaces by considering $A\in\mathcal{P}_\kappa(S_X)$, instead.
\end{definition}

Notice that, if $\kappa$ is infinite, then we can equivalently require just that $\|x + y\|\le1+\e$. Moreover, a standard argument shows that, for the SQ case, it is equivalent to require that $\|x \pm y\|=1$.

It is known that a Banach space $X$ is ASQ if and only if, for every finite-dimensional subspace $Y\subset X$ and $\varepsilon>0$, there is $x\in S_X$ such that
\begin{equation*}
    \|y+ rx\|\le(1+\e)(\|y\|\vee|r|)\text{ for all }y\in Y\text{ and }r\in\mathbb R
\end{equation*}
(see \cite[Proposition~2.1]{ALL2016}). Even more is true, in fact, a straightforward application of \cite[Lemma 2.2]{ALL2016} provides a description of the above notions via subspaces of density $<\kappa$, whenever $\kappa$ is uncountable. The version for $\kappa=\aleph_0$ is unknown for the authors.

\begin{proposition}\label{prop:chartransfASQ}
    Let $X$ be a Banach space and $\kappa$ an uncountable cardinal. The following are equivalent:
    \begin{itemize}
        \item[(i)] $X$ is $\ASQ{\kappa}$ ($\SQ{\kappa}$, respectively).
        \item[(ii)] For every subspace $Y\subset X$ with $\dens(Y)<\kappa$ and $\varepsilon>0$ ($\e \geq 0$, respectively), there exists $x\in S_X$ such that
    $$\Vert y+r x\Vert\leq (1+\varepsilon)(\Vert y\Vert\vee\vert r\vert)$$
    holds for every $y\in Y$ and every $r\in\mathbb R$.
    \end{itemize}
\end{proposition}

Let us devote the rest of the section to provide various examples of transfinite (A)SQ spaces.

\begin{example}\label{example:countablysupported} Let $\kappa$ be an uncountable cardinal, and let $\ell_\infty^c(\kappa)$ be the elements of $\ell_\infty(\kappa)$ whose support is at most countable. If $X$ is a subspace of $\ell_\infty^c(\kappa)$ containing $c_0(\kappa)$, then $X$ is $\SQ{\kappa}$. Indeed, fix $A\in\mathcal{P}_{<\kappa}(S_X)$. Since $\supp(f)\subseteq\kappa$ is at most countable for every $f\in A$, we can find $\lambda\in\kappa$ so that $\lambda\notin \bigcup_{f\in A}\supp(f)$. Clearly $\Vert f+e_\lambda\Vert=1$ holds for every $f\in A$.
\end{example}

\begin{example}\label{ex: ell_infty rigid ASQ}
    Fix a non-principal ultrafilter $\mathcal U$ in $\mathbb N$. For every $x\in\ell_\infty$, denote by $\lim(x)$ the limit of $x(n)$ with respect to $\mathcal U$ and define the norm
    \begin{equation*}
        \nn x:=|\lim(x)|\vee\bigvee_{n\in\mathbb N}|x(n)-\lim(x)|.
    \end{equation*}
    The Banach space $X:=(\ell_\infty,\nn\cdot)$ was defined in \cite{BLR2016} and proved to be ASQ. In the following we prove that it actually is $\SQ{\aleph_0}$. Nevertheless, $X$ cannot be ASQ$_{\aleph_0}$ (see Theorem~\ref{thm: C(K) non renormable to be ASQ}) since $\ell_\infty=C(\beta\mathbb N)$.
    
    Fix $x_1,\ldots,x_k\in S_X$ and define, for every $n\in\mathbb N$ and $m\in\{1,\ldots,k\}$, the set
    \begin{equation*}
        A_{n,m}:=\{p\in\mathbb N:|x_m(p)-\lim(x_m)|<1/n\}.
    \end{equation*}
    By definition of ultralimit, $A_{n,m}\in\mathcal U$, therefore $A_n:=\bigcap_{m=1}^kA_{n,m}\in\mathcal U$ for every $n\in\mathbb N$. Since $\mathcal U$ is non-principal, each $A_n$ is infinite, hence, for every $n\in\mathbb N$, we can find $f(n)\in A_n$ such that $f(n)<f(n+1)$. Notice that, since $\emptyset\notin\mathcal U$, either $f(2\mathbb N)$ or $f(2\mathbb N+1)$ is not in $\mathcal U$, say for example that $f(2\mathbb N)\notin\mathcal U$. Define the formal series
    \begin{equation*}
        y:=\sum_{n\in 2\mathbb N}(1-1/n)e_{f(n)}\in\ell_\infty.
    \end{equation*}
    Notice that $\|y\|_\infty=1$ and $\lim(y)=0$, therefore $y\in S_X$. For every $i\in\{1,\ldots,k\}$, notice that
    \begin{equation*}
        \nn{x_i+ y}=|\lim(x_i)|\vee\bigvee_{n\in\mathbb N}|x_i(n)-\lim(x_i)+ y(n)|\le
    \end{equation*}
    \begin{equation*}
        \le1\vee\bigvee_{n\in 2\mathbb N}|x_i(f(n))-\lim(x_i)+ (1-1/n)|\vee\bigvee_{n\in \mathbb N\setminus f(2\mathbb N)}|x_i(n)-\lim(x_i)|\le
    \end{equation*}
    \begin{equation*}
        \le1\vee(1-1/n+1/n)\vee1=1.
    \end{equation*}
Therefore $X$ is $\SQ{\aleph_0}$.
\end{example}

The previous example yields a non-separable example of a $\SQ{\aleph_0}$ space. A natural question at this point is whether or not there is a separable Banach space which is $\SQ{\aleph_0}$. The next example is a modification of \cite[Example~6.4]{ALL2016} and provides an affirmative answer.

\begin{example}
Given $n\in\mathbb N$, consider
$$X_n:=\{f\in C(S_{\bb R^n}): f(s)=-f(-s)\mbox{ for all }s\in S_{\bb R^n}\}.$$

Let us show that $X_n$ is SQ$_n$. Fix $f_1, \dots, f_n \in S_{X_n}$.
By a corollary of Borsuk--Ulam theorem \cite[p.~485, Satz VIII]{AH1974}, we can find $s_0\in S_{\bb R^n}$ such that $f_i(s_0)=0$ holds for every $i\in\{1,\ldots,n\}$. Pick any function $h\in S_{X_n}$ such that $h(s_0)=1$ and define
\[
    g(s):=\left(1-\bigvee_{i=1}^n|f_i(s)|\right)h(s).
\]
Notice that $g\in X_n$ and that $g(s_0)=1$, therefore $\|g\|=1$. For every $i\in\{1,\ldots,n\}$ and $s\in S_{\bb R^n}$ we have
\[
    |f_i(s)\pm g(s)|\le|f_i(s)|+|g(s)|\le|f_i(s)|+1-\bigvee_{j=1}^n|f_j(s)|\le1,
\]
as required.

Now define $X:=c_0(\mathbb N, X_n)$. It is obvious that $X$ is separable. Moreover, since $X_n$ is SQ$_n$ for every $n \in \bb N$, it is immediate to check that $X$ is SQ$_{<\aleph_0}$. In fact, fix $x_1,\ldots,x_k\in S_{X}$ and without loss of generality assume that $\|x_i(k)\|=1$ for all $i\in\{1,\ldots,k\}$. Find $y\in S_{X_k}$ such that $\|x_i(k)+y\|\le1$ holds for all $i\in\{1,\ldots,k\}$, therefore $\|x_i+y\cdot e_k\|\le1$.
\end{example}

Spaces of (almost) universal disposition were introduced in \cite{G1966}. Let us recall their definitions. Given a family of Banach spaces $\mathfrak K$, a Banach space $X$ is \emph{of almost universal disposition for $\mathfrak K$} if for every $S\subset T$ in $\mathfrak K$, any isometric embedding $f:S\rightarrow X$ extends to an $\varepsilon$-isometric embedding $F:T\rightarrow X$. Moreover, a Banach space $X$ is \emph{of universal disposition for $\mathfrak K$} if for every $S\subset T$ in $\mathfrak K$, any isometric embedding $f:S\rightarrow X$ extends to an isometric embedding $F:T\rightarrow X$.

\begin{example}
    If $X$ is of almost universal disposition (of universal disposition, respectively) for Banach spaces with density character strictly less then $\kappa$, then $X$ is $\ASQ{\kappa}$ ($\SQ{\kappa}$, respectively). We show the claim for the ASQ case only. Fix a subspace $Y\subset X$ with $\dens(Y)<\kappa$ and $\varepsilon>0$. The inclusion $Y\rightarrow X$ extends to an $\varepsilon$-isometrical embedding $T:Y\oplus_\infty\mathbb R\rightarrow X$. Find $r\in\mathbb R$ such that $\|T(0,r)\|=1$. We can do so since $T$ is injective by picking any $s\not=0$ and setting $r:=s/T(0,s)$. Notice that
    \begin{equation*}
        |r|=\|(0,r)\|_\infty\le(1+\varepsilon)\|T(0,r)\|=1+\varepsilon.
    \end{equation*}
    It is clear that, for every $y\in S_Y$,
    \begin{equation*}
        \|y+T(0,r)\|\le(1+\varepsilon)\|(y,0)+(0,r)\|_\infty=(1+\varepsilon)(\|y\|\vee|r|)\le(1+\varepsilon)^2.
    \end{equation*}
\end{example}

In the following we study $C_0(X)$ spaces. It is known that, given a locally compact Hausdorff space $X$, $C_0(X)$ is ASQ if and only if $X$ is non-compact \cite[Proposition~2.1]{AMR2022}. In the following we provide a topological description of $X$ so that $C_0(X)$ is $\ASQ{\kappa}$, whenever $\kappa$ is uncountable and, as a byproduct, we find out that being $\SQ{\kappa}$ and $\ASQ{\kappa}$ are equivalent in $C_0(X)$ spaces, at least under a mild regularity assumption on $X$.

\begin{theorem}\label{thm: C_0(X) spaces}
    Let $X$ be a T$_4$ locally compact space. If $\kappa$ is an uncountable cardinal, then the following are equivalent:
    \begin{itemize}
        \item[(i)] $C_0(X)$ is $\SQ{\kappa}$,
        \item[(ii)] $C_0(X)$ is $\ASQ{\kappa}$,
        \item[(iii)] If $\mathscr K\in\mathcal{P}_{<\kappa}(\mathcal{P}(X))$ is a family consisting of compact sets in $X$, then $\bigcup\mathscr K$ is not dense in $X$.
    \end{itemize}
\end{theorem}
\begin{proof}
    (i)$\implies$(ii) is obvious. (ii)$\implies$(iii). Fix a family $\mathscr K\in\mathcal{P}_{<\kappa}(\mathcal{P}(X))$ consisting of compact sets in $X$ and fix any $K\in\mathscr K$. Since $K$ is compact and $X$ is locally compact, we can find a covering $U_1,\ldots,U_n$ for $K$ consisting of open relatively compact sets. Define $U:=\bigcup_{i=1}^nU_i$ and notice that $X\setminus U\not=\emptyset$, otherwise we would get that $X=\overline{U}$, which is compact, and this would contradict the fact that $C_0(X)$ is $\ASQ{\kappa}$. On the other hand, it is clear that $K$ and $X\setminus U$ are disjoint closed sets, therefore, since $X$ is normal, there exists a Urysohn's function $f_K:X\rightarrow[0,1]$ such that $f_K|_K=1$ and $f_K|_{X\setminus U}=0$. Notice that the support of $f_K$ is contained in $\overline{U}$, which is compact, thus $f_K\in S_{C_0(X)}$. Since $C_0(X)$ is $\ASQ{\kappa}$, there is $g\in S_{C_0(X)}$ satisfying
    \begin{equation*}
        \|f_K\pm g\|_\infty\le3/2\text{ for every }K\in\mathscr K.
    \end{equation*}
    It is clear by construction that $|g(x)|\le1/2$ holds for every $x\in\bigcup\mathscr K$. Therefore the non-empty open set $\{x\in X:|g(x)|>1/2\}$ is disjoint from $\bigcup\mathscr K$, hence $\bigcup\mathscr K$ is not dense in $X$.
    
    (iii)$\implies$(i). Fix $A\in\mathcal{P}_{<\kappa}(S_{C_0(X)})$. For every $f\in A$ and $n\in\mathbb N$, there exists a compact set $K_{f,n}\subset X$ such that $|f(x)|<1/n$ holds for every $x\in X\setminus K_{f,n}$. Define $\mathscr K:=\{K_{f,n}:f\in A\text{ and }n\in\mathbb N\}$ and notice that $|\mathscr K|\le|A|\cdot\aleph_0<\kappa$ since $\kappa$ is uncountable. By assumption we can find a non-empty open set $U$ which is disjoint from $\bigcup\mathscr K$ and, without loss of generality, we can assume it to be relatively compact. Since $X$ is normal, there exists a Urysohn's function $g:X\rightarrow[0,1]$ such that $\|g\|_\infty=1$ and $g|_{X\setminus U}=0$. Notice that the support of $g$ is contained in $\overline{U}$, which is compact, thus $g\in S_{C_0(X)}$. It is clear by construction that $\|f+g\|_\infty=1$ holds for every $f\in A$.
\end{proof}

A closer look to the proof of Theorem~\ref{thm: C_0(X) spaces} reveals that  (ii)$\iff$(iii)  actually holds without the assumption that $\kappa$ is uncountable. This corresponds to the already recalled result that $C_0(X)$ is ASQ if and only if $X$ is non-compact.  

Let us note that in the case $\kappa = \aleph_1$, the property (iii) of Theorem~\ref{thm: C_0(X) spaces} can be stated as ``$X$ does not admit a dense sigma-compact set''.

\section{Banach spaces which admit transfinite ASQ renorming}\label{section:renormings}

Let $X$ be a Banach space. In this section we will analyse the relations between the following properties:
\begin{itemize}
    \item[(a)] $X$ admits an equivalent $\ASQ{\kappa}$ renorming,
    \item[(b)] $X$ admits an equivalent $\SQ{\kappa}$ renorming,
    \item[(c)] $X$ contains an isomorphic copy of $c_0(\kappa)$.
\end{itemize}

Observe that an easy transfinite induction argument reveals that, if $X$ admits an equivalent $\SQ{\kappa}$ renorming, then $X$ contains an isomorphic copy of $c_0(\kappa)$. The situation is dramatically different if we replace the SQ norm with an ASQ norm.

\begin{example}\label{theo:kappaASQnotconc0}
Let $\kappa$ be an infinite cardinal. Then $X:=c_0(\mathbb{N}_{\geq 2}, \ell_n(\kappa))$ is $\ASQ{\kappa}$ but $X^*$ does not contain any isomorphic copy of $\ell_1(\omega_1)$, in particular, $X$ does not contain any isomorphic copy of $c_0(\omega_1)$.
\end{example}

\begin{proof}
Let us initially suppose that $\kappa>\aleph_0$. In order to prove that $X$ is $\ASQ{\kappa}$, it is enough to note that, 
for every set $A \in\mathcal{P}_{<\kappa}(S_{\ell_n(\kappa)})$, there is $y \in S_{\ell_n(\kappa)}$ such that $\|x+y\| \leq 2^{\frac 1 n}$ for all $x \in A$ (later, in Section~\ref{section:generalisations}, we will call this property $\rsSQ{2^{-\frac 1 n}}{2^{-\frac 1 n}}{\kappa}$, see Definition~\ref{def: parametric ASQ}). Indeed, since every $x \in A$ has a countable support, we can find $\lambda \in \kappa$ such that $x(\lambda) = 0$ for all $x \in A$.
Take $y$ defined by $y(\mu) = \delta_{\lambda\mu}$. Now
$\|x+ y\| = (\|x\|^n + \|y\|^n)^{\frac 1 n} = 2^{\frac 1 n}$ for all $x \in A$, as required.
In order to conclude, fix $A\in\mathcal{P}_{<\kappa}(S_X)$ and $\varepsilon>0$. Find $n\in\mathbb N$ such that $2^{\frac{1}{n}}<1+\varepsilon$ and find $y\in S_{\ell_n(\kappa)}$ such that $\|x(n)+y\|\le2^\frac{1}{n}$ holds for all $x\in A$, then $\|x+y\cdot e_n\|\le1+\varepsilon$.

If $\kappa = \aleph_0$, the proof is similar, but, for every $\e > 0$, we can only manage to find $y$ such that
$\|x + y\| \leq (2+\e)^{\frac 1 n}$, which is still enough. Indeed, given a finite set $A \subset S_{\ell_n}$, we can find $m \in \bb N$ such that $|x(m)| < \delta$ for all $x \in A$, where $\delta > 0$ is chosen such that $(1+\delta)^n < 1+\e$, thus we only need to define $y:=e_m$.

In order to prove the second part, observe that $X^*=\ell_1(\mathbb{N}_{\geq 2},\ell_{n^*}(\kappa))$ where $n^*$ is the conjugate index of $n$. Since $X^*$ is a countable sum of reflexive Banach spaces, we derive that $X^*$ is weakly compactly generated. Consequently $X^*$ cannot contain $\ell_1(\omega_1)$, which even fails weaker properties (e.g. it is immediate to see that $\ell_1(\omega_1)$ fails the Corson property (C) by using \cite[Theorem 12.42]{FHHMPZ2001}, which is inherited by closed subspaces).

Finally, to conclude that $X$ does not contain $c_0(\omega_1)$, observe that if $X$ contained $c_0(\omega_1)$, taking adjoint we would have that $\ell_1(\omega_1)$ would be isomorphic to a quotient of $X^*$. Since $\ell_1(\omega_1)$ has the lifting property, we would conclude that $\ell_1(\omega_1)$ is isomorphic to a subspace of $X^*$, which entails a contradiction with the previous point.
\end{proof}

\begin{remark}\label{rem: ell_infty(ell_n) is ASQ}
The same proof as in Example~\ref{theo:kappaASQnotconc0} gives that $\ell_\infty(\mathbb{N}_{\geq 2},\ell_n(\kappa))$ also is $\ASQ{\kappa}$. More in detail, using the terminology from Example~\ref{theo:kappaASQnotconc0}, since each $\ell_n(\kappa)$ is $\rsSQ{2^{-\frac 1 n}}{2^{-\frac 1 n}}{\kappa}$, the claim follows from a direct computation or from Theorem~\ref{thm: ell_infty sum}. This shows that there are dual (actually bidual) Banach spaces which are $\ASQ{\kappa}$, for every infinite cardinal $\kappa$. Let us point out that the importance of this result is that, for classical ASQ spaces, it was posed in \cite{ALL2016} the question whether there is any dual ASQ space, which was positively solved in \cite{AHT2020}.

Observe that the situation for SQ spaces is different, because they are clearly incompatible with the existence of extreme points in the unit ball, so no dual Banach space can enjoy any SQ property.
\end{remark}

We have seen that the $\ASQ{\kappa}$ condition does not imply the containment of large copies of $c_0$. However, this behaviour is impossible in spaces of continuous functions, as the following theorem shows.

\begin{theorem}\label{thm: C(K) non renormable to be ASQ}
Let $K$ be a compact Hausdorff topological space. If $C(K)$ admits any equivalent $\ASQ{\kappa}$ norm, then it contains an isomorphic copy of $c_0(\kappa)$.
\end{theorem}

\begin{proof}
    Call $X:=C(K)$ and assume that $$\frac{1}{M}\Vert f\Vert\leq \nn f\leq M\Vert f\Vert$$
    holds for every $f\in X$. Find $p\in\mathbb N$ big enough and $\varepsilon>0$ small enough so that $p>2M^2(1+\varepsilon)^p$. If $(X,\nn\cdot)$ is $\ASQ{\kappa}$, then, by transfinite induction, we can find $\{f_\alpha:\alpha<\kappa\}\subseteq S_{(X,\nn{\cdot})}$ so that
    $$\nn{f+r f_\alpha}\leq (1+\varepsilon)(\nn f\vee\vert r\vert)$$
    holds for every $f\in \overline{\spn}\{f_\beta:\beta<\alpha\}$ and $r\in\mathbb R$. Note that $\Vert f_\alpha\Vert\geq \frac{1}{M}$. Up to consider $-f_\alpha$ instead, we can assume that the set
    $$V_\alpha:=\left\{x\in K: f_\alpha(x)>\frac{1}{2M}\right\}$$
    is non-empty, and it is clearly open. Suppose by contradiction that $c_0(\kappa)$ does not embed in $C(K)$, we then derive that $K$ satisfies the $\kappa$-chain condition \cite[page 227]{Rosenthal1970}. By \cite[Page 227, Remark]{Rosenthal1970} we get that, since $\{V_\alpha: \alpha<\kappa\}$ is a family of open sets in $K$, we can find an infinite set $\{\alpha_n: n\in\mathbb N\}\subseteq [0,\kappa)$ so that there exists $x\in \bigcap\limits_{n\in\mathbb N}V_{\alpha_n}$. Eventually,
    \[
    \begin{split}
        (1+\varepsilon)^p>\nn{\sum_{i=1}^p f_i}\geq \frac{1}{M}\left\Vert\sum_{i=1}^p f_i \right\Vert\geq \frac{1}{M}\sum_{i=1}^p f_i(x)\geq \frac{p}{2M^2},
    \end{split}
    \]
    which is a contradiction.
\end{proof}

Now it is time to analyse the following question.

\begin{question}
Let $\kappa$ be an infinite cardinal. If a Banach space contains an isomorphic copy of $c_0(\kappa)$, does it admit an equivalent $\SQ{\kappa}$ renorming?
\end{question}

We do not know the answer to the above question in complete generality. However, we are able to give some partial positive answers. The first one deals with Banach spaces that contains $c_0$.

\begin{proposition}\label{prop:rigidASQdual}
If $X$ is a dual Banach space containing an isomorphic copy of $c_0$, then $X$ admits an equivalent $\SQ{\aleph_0}$ renorming.
\end{proposition}

\begin{proof}
    If $X$ is a dual Banach space containing $c_0$, then $X$ contains an isomorphic copy of $\ell_\infty$ \cite[Proposition 2.e.8]{LiTza}. Because of its injectivity, $\ell_\infty$ is complemented in $X$ (c.f. e.g. \cite[Proposition 5.13]{FHHMPZ2001}). Consequently, there is a subspace $Z$ of $X$ so that $X=\ell_\infty \oplus Z$. Consider the norm $\nn\cdot$ on $\ell_\infty$ described in Example \ref{ex: ell_infty rigid ASQ}. Now, consider on $X$ the equivalent norm so that $X=(\ell_\infty,\nn\cdot)\oplus_\infty Z$. Then $X$, endowed with this norm, is $\SQ{\aleph_0}$, because  $(\ell_\infty,\nn\cdot)$ is $\SQ{\aleph_0}$ together with Corollary \ref{cor: infty sum}.
\end{proof}

 We do not know whether $c_0$ has an equivalent $\SQ{\aleph_0}$ renorming and, when $\kappa>\aleph_0$, we do not know if $\ell_\infty(\kappa)$ has an equivalent $\SQ{\kappa}$ renorming. The best that we can say in this direction is the following.

\begin{proposition}
Let $\kappa$ and $\lambda$ be uncountable cardinals. If there exists a $\kappa$-complete ultrafilter $\mathcal{U}$ on $\lambda$. Then $\ell_\infty(\lambda)$ admits an equivalent $SQ_{<\kappa}$ renorming. 
\end{proposition}

\begin{proof}
    Define an equivalent norm by the same formula as in Example \ref{example:countablysupported}
    $$\nn x:=|\lim(x)|\vee\bigvee_{\mu\in\lambda}|x(\mu)-\lim(x)|,$$
    where $\lim(x)$ denotes the limit through the ultrafilter $\mathcal{U}$. Similarly as before, if we take $X\in\mathcal{P}_{<\kappa}(S_{\ell_\infty(\lambda)})$, then the sets $$A_{n,x} = \{\mu\in\lambda : |x(\mu) - \lim(x)|<1/n\}\in\mathcal U$$
    for all $n\in\mathbb{N}$ and $x\in X$. So $A:=\{\mu\in\lambda : x(\mu) = \lim(x)\} = \bigcap_{n,x}A_{n,x}\in\mathcal{U}$. If we take $\mu\in A$, then it is easily checked that $\nn{x+e_\mu} = 1$ holds for all $x\in X$.
\end{proof}

This statement is quite unsatisfactory because $\lambda$ must be a large cardinal, at least the first measurable cardinal. Using a variation of this idea using multiple ultrafilters instead of just a fixed one, we obtain another general result which says that, when $X$ contains $c_0(\kappa)$ and $X/c_0(\kappa)$ is somehow small, then $X$ admits an equivalent $\SQ{\kappa}$ norm. This is the main result of this section.

\begin{theorem}\label{theo:renormingkASQ}
Let $\kappa$ be an infinite cardinal of uncountable cofinality. If a Banach space of density character $\kappa$ contains an isomorphic copy of $c_0(\kappa)$, then it admits an equivalent $\SQ{\mathrm{cf}(\kappa)}$ renorming.	
\end{theorem}
 
\begin{proof}
	Without loss of generality we can suppose that the copy of $c_0(\kappa)$ is isometric. Let $Y\subset X$ be a subspace together with an isometric isomorphism $S:Y\longrightarrow c_0(\kappa)$.
	By Hahn-Banach, there exists an norm-1 operator $T:X\longrightarrow \ell_\infty(\kappa)$ such that $T|_Y = S$.
	
	Now we aim to define a suitable one-to-one mapping $g:\kappa\longrightarrow B_{X^{^\ast}}$ such that all $g(\alpha)$'s vanish on $Y$. After doing so, we define the equivalent norm
	
	$$\nn x := \|x\|_{X/Y} \vee \bigvee_{\alpha<\kappa}|T_\alpha(x)-g(\alpha)(x)| $$ 
	
	\textbf{First step:} $\nn\cdot$ defines an equivalent norm on $X$.
	
	In fact, it is clear that $\nn\cdot$ is a norm and that $\nn\cdot\leq 2 \|\cdot\|$. Now suppose by contradiction that that we cannot obtain the opposite inequality with respect to any fixed constant, then we can find a sequence $(x_n)_{n\in\mathbb N}\subset S_X$ satisfying $\lim_n \nn{x_n}= 0$. This implies that $\lim_n \|x_n\|_{X/Y} = 0$, so we can find elements $y_n\in Y$ such that $\lim_n\|x_n-y_n\|=0$. This, as before, shows that $\lim_n\nn{x_n-y_n}=0$.
	
	Since $\lim_n \nn{x_n}= 0$, we conclude that $\lim_n\nn{y_n}=0$, but, since $y_n\in Y$, then $\nn{y_n}= \bigvee_{\alpha<\kappa}|T_\alpha(y_n)| = \|y_n\|$, hence $\lim_n\|y_n\|=0$. This, together with $\lim_n \|x_n-y_n\|=0$, implies that $\lim_n\|x_n\|=0$, which is a contradiction.\vspace{3mm}

\textbf{Second step:} If $T_\beta(x) = g(\beta)(x)$, then $\nn{x+tS^{-1}(e_\beta)}= \nn x\vee |t|$.

In fact, call $u_\beta := S^{-1}(e_\beta)$ and observe that
$$\nn x = \|x\|_{X/Y} \vee \bigvee_{\alpha\neq\beta}|T_\alpha(x)-g(\alpha)(x)|.$$
Thanks to the fact that $g(\beta)(u_\beta)=0$, we deduce that $|T_\beta(x+tu_\beta)-g(\beta)(x+tu_\beta)|=|t|$. Notice also that $T_\alpha(u_\beta) = g(\alpha)(u_\beta)=0$. Therefore
\begin{align*}
    \nn{x+tu_\beta} & = \|x\|_{X/Y} \vee |T_\beta(x+tu_\beta)-g(\beta)(x+tu_\beta)| \vee \bigvee_{\alpha\neq\beta}|T_\alpha(x)-g(\alpha)(x)|=\\
    & =\nn x\vee|t|.
\end{align*}

\textbf{Third step:} If $Z\subset X$ is a subspace with $\dens(Z)<\kappa$, then, for all $\alpha\in\kappa$, except for $<\kappa$ many $\alpha$'s, there exist functionals $g_\alpha\in B_{X^*}$ that vanish on $Y$ and such that $T_\alpha(x)=g_\alpha(x)$ holds for all $x\in Z$.

In fact, consider the continuous function $\phi: \beta\kappa \longrightarrow B_{Z^*}$ given by 
\begin{equation*}
    \phi(\mathcal{U})(z)= \lim_{\mathcal{U}}T_\gamma(z),
\end{equation*}
where the limit is taken with respect to $\gamma$, and the topology on $B_{Z^*}$ is the weak$^*$ topology. Notice that, for $\alpha<\kappa$, if a non-principal ultrafilter $\mathcal{U}_\alpha$ in $\kappa$ satisfies that $\phi(\mathcal{U_\alpha}) = \phi(\alpha)$, then $g_\alpha:=\lim_{\mathcal{U}_\alpha}T_\gamma$ satisfies the desired conditions. So it is enough to show that for all but less than $\kappa$ many $\alpha<\kappa$ such non-principal ultrafilter exists. Suppose for contradiction, that this is not the case. This means that there is a set $A\subset \kappa$ of cardinality $\kappa$ such that $\phi^{-1}\{\phi(\alpha)\}$ contains no non-principal ultrafilters for all $\alpha\in A$. This means that $\phi^{-1}\{\phi(\alpha)\}$ consists only of isolated points of $\beta\kappa$, but it is also a compact set by continuity. Hence each set $\phi^{-1}\{\phi(\alpha)\}$ is finite, for $\alpha\in A$. This implies that $\{\phi(\alpha) : \alpha\in A\}$ has cardinality $\kappa$. Now we prove that each point $\phi(\alpha)$ is an isolated point of the range $\phi(\beta\kappa)\subset B_{Z^*}$. This is a contradiction with the fact that $B_{Z^*}$ has weight less than $\kappa$ since $Z$ had density less than $\kappa$.  So suppose that $\phi(\alpha)$ is not isolated in that range. Since $\kappa$ is dense in $\beta\kappa$, we must have $$\phi(\alpha)\in \overline{\{\phi(\beta) : \beta<\kappa, \phi(\beta)\neq\phi(\alpha)\}}.$$
Consider
$$\mathcal{F} = \{B\subset \kappa : \exists W\text{ neighborhood of }\phi(\alpha) : \kappa\cap \phi^{-1}(W\setminus\{\phi(\alpha)\})\subset B \}.$$
This is a filter of subsets of $\kappa$ that contains all complements of finite sets, and satisfies $\{\phi(\alpha)\} = \bigcap_{B\in\mathcal{F}}\overline{\phi(B)}$. There is a non-principal ultrafilter $\mathcal{U}$ on $\kappa$ that contains $\mathcal{F}$, and we have $$\phi(\mathcal{U}) = \phi(\lim_\mathcal{U}\beta) = \lim_\mathcal{U}\phi(\beta)  = \phi(\alpha).$$
This contradicts that $\phi^{-1}\{\alpha\}$ contained no non-principal ultrafilter.\vspace{3mm}



\textbf{Fourth step:} Definition of the map $g:\kappa\longrightarrow B_{X^{^\ast}}$.

Let $\{X_\gamma : \gamma<\text{cf}(\kappa)\}$ be a family consisting of subspaces of $X$ of density character strictly less than $\kappa$, such that every subspace of $X$ with density character strictly less than $\text{cf}(\kappa)$ is contained in some $X_\gamma$. Using the previous step, for each $\gamma<\text{cf}(\kappa)$, we can inductively choose $\alpha(\gamma)<\kappa$ and $g_\gamma\in B_{X^*}$, such that $g_\gamma$ vanishes on $Y$, $T_{\alpha(\gamma)}|_{X_\gamma} = g_{\alpha(\gamma)}|_{X_\gamma}$ and $\alpha(\gamma')\neq \alpha(\gamma)$ for $\gamma'<\gamma$. It is now legitimate to define 
$$g(\alpha):=\bigg\{\begin{array}{ll}
    g_{\alpha(\gamma)} & \text{ if }\alpha= \alpha(\gamma)\text{ for some }\gamma<\text{cf}(\kappa),\\
    0 & \text{ if }\alpha\not\in\{\alpha(\gamma): \gamma<\text{cf}(\kappa)\}
\end{array}$$\vspace{1mm}

Finally, we can conclude the proof of the theorem. For this purpose, fix a subspace $Z\subset X$ with $\dens(Z)<\text{cf}(\kappa)$ and find $\gamma<\text{cf}(\kappa)$ such that $Z\subset X_\gamma$. By construction, $T_{\alpha(\gamma)}(x)=g_{\alpha(\gamma)}(x)$ holds for all $x\in X_\gamma$. Thanks to the second step, we can find an element $y\in S_{(X,\nn\cdot)}$ such that $\nn{x+ty}\le\nn{x}\vee|t|$ holds for all $x\in X_\gamma$ and $t\in\mathbb R$.
\end{proof}
 
As an application of the above results we get the following.

\begin{corollary}
	$\ell_\infty/c_0$ admits a $\SQ{\mathrm{cf}(\mathfrak{c})}$ equivalent norm. In particular, a SQ$_{\aleph_0}$ equivalent norm.
\end{corollary}

\begin{proof}
$\ell_\infty/c_0$ contains a subspace isometric to $c_0(\mathfrak{c})$, coming from an almost disjoint family of cardinality $\mathfrak{c}$ and $\text{cf}(\mathfrak{c})>\aleph_0$.
\end{proof}

\section{Stability results}\label{section:stabilityresults}

In this section we aim to produce more examples of Banach spaces which are transfinite (A)SQ.

\subsection{Direct sums}

It is known that the only possible sums which may preserve ASQ are the $c_0$ and the $\ell_\infty$ sums \cite[Theorem~3.1]{HLLN2019}. Because of that, we will only focus on these two cases.

\begin{theorem}\label{thm: ell_infty sum}
    Let $\{X_\alpha:\alpha\in\mathscr A\}$ be a family of Banach spaces and $\kappa$ an infinite cardinal. If, for every $\varepsilon>0$, there exists $\beta\in\mathscr A$ such that, for every set $A\in\mathcal{P}_\kappa(S_{X_\beta})$, there is $y\in S_{X_\beta}$ satisfying
    \begin{equation*}
        \|x+ y\|\le1+\varepsilon\text{ for all }x\in A,
    \end{equation*}
    then $\ell_\infty(\mathscr A,X_\alpha)$ is $\ASQ{\kappa}$. Moreover, if $|\mathscr A|<\text{cf}(\kappa)$ and $\lambda^{|\mathscr A|}<\kappa$ holds for every cardinal $\lambda<\kappa$, then the converse holds too.
\end{theorem}
\begin{proof}
    Fix $A\in\mathcal{P}_{<\kappa}(S_{\ell_\infty(\mathscr A,X_\alpha)})$ and $\varepsilon>0$. Find $\beta\in\mathscr A$ as in the assumption. Then there exists $y\in S_{X_{\beta}}$ satisfying
    \begin{equation*}
        \|x(\beta)+ y\|\le(1+\varepsilon)(\|x(\beta)\|\vee 1)=1+\varepsilon\text{ for all }x\in A.
    \end{equation*}
    We conclude that
    \begin{equation*}
        \|x+ y\cdot e_\beta\|_\infty=\bigvee_{\alpha\in\mathscr A\setminus\{\beta\}}\|x(\alpha)\|\vee\|x(\beta)+ y\|\le 1\vee (1+\varepsilon)=1+\varepsilon
    \end{equation*}
    holds for every $x\in A$. Hence $\ell_\infty(\mathscr A,X_\alpha)$ is $\ASQ{\kappa}$.
    
    For the additional part, suppose that $\ell_\infty(\mathscr A,X_\alpha)$ is $\ASQ{\kappa}$ and, by contradiction, that there exists $\varepsilon>0$ such that for every $\alpha\in\mathscr A$ there exists a set $A_\alpha\in\mathcal{P}_{<\kappa}(S_{X_\alpha})$ such that for every $y\in S_{X_\alpha}$ there is $x\in A_\alpha$ satisfying
    \begin{equation*}
        \text{either }\|x+y\|>1+\varepsilon\text{ or }\|x-y\|>1+\varepsilon.
    \end{equation*}
    Notice that
    \begin{equation*}
        \left|\prod_{\alpha\in\mathscr A}A_\alpha\right|\le\left(\sup_{\alpha\in\mathscr A}|A_\alpha|\right)^{|\mathscr A|}<\kappa,
    \end{equation*}
    where the last inequality follows from observing that $\sup_{\alpha\in\mathscr A}|A_\alpha|<\kappa$ since $|\mathscr A|<\text{cf}(\kappa)$ and $\lambda^{|\mathscr A|}<\kappa$ hold by hypothesis for every cardinal $\lambda<\kappa$. Since $\ell_\infty(\mathscr A,X_\alpha)$ is $\ASQ{\kappa}$, we can find $y\in S_{\ell_\infty(X_\alpha)}$ such that
    \begin{equation*}
        \|x+ y\|_\infty\le1+\varepsilon/2\text{ for every }x\in\prod_{\alpha\in\mathscr A}A_\alpha.
    \end{equation*}
    Find $\beta\in\mathscr A$ such that $\|y(\beta)\|\ge1-\varepsilon/2$. Then, for every $x\in \prod_{\alpha\in\mathscr A}A_\alpha$, we get
    \[\begin{split}
        1+\varepsilon/2\ge\|x+ y\|_\infty\ge\|x(\beta)+ y(\beta)\|& \ge\left\|x(\beta)+\frac{y(\beta)}{\|y(\beta)\|}\right\|-\left\|y(\beta)-
        \frac{y(\beta)}{\|y(\beta)\|}\right\|\\
        & \ge\left \|x(\beta)+\frac{y(\beta)}{\|y(\beta)\|}\right\|-\varepsilon/2.
    \end{split}\]
    This implies that $\|x(\beta)+ y(\beta)/\|y(\beta)\|\|\le1+\varepsilon$, which is a clear contradiction since it holds for every $x(\beta)\in A_\beta$.
\end{proof}

\begin{corollary}\label{cor: infty sum}
    Let $X$ and $Y$ be Banach spaces and $\kappa$ an infinite cardinal. Then $X\oplus_\infty Y$ is $\AandSQ{\kappa}$ if and only if either $X$ or $Y$ is $\AandSQ{\kappa}$.
\end{corollary}
\begin{proof}
    Notice that, with the notation from the statement of Theorem~\ref{thm: ell_infty sum}, $|\mathscr A|=2<\text{cf}(\kappa)$ and that $\lambda^{|A|}=\lambda\cdot\lambda=\lambda<\kappa$ holds for every $\lambda<\kappa$. Therefore we can apply both directions of Theorem~\ref{thm: ell_infty sum}, hence the ASQ case follows. For the SQ case, notice that the first half of the proof of Theorem~\ref{thm: ell_infty sum} holds also if we choose $\varepsilon=0$. Moreover, the second half of the proof of Theorem~\ref{thm: ell_infty sum} holds for $\varepsilon=0$ too if $\mathscr A$ is finite.
\end{proof}

It is know that, given any sequence of Banach spaces $\{X_n:n\in\mathbb N\}$, the Banach space $c_0(\mathbb N,X_n)$ is always ASQ \cite[Example~3.1]{ALL2016}. A transfinite generalisation of this result is the following.

\begin{proposition}
    Let $\{X_\alpha:\alpha\in\mathscr A\}$ be an uncountable family of Banach spaces. Then the Banach space $c_0(\mathscr A,X_\alpha)$ is $\SQ{\vert \mathscr A \vert}$.
\end{proposition}
\begin{proof}
    Fix $A\in\mathcal{P}_{<|\mathscr A|}(S_{c_0(\mathscr A,X_\alpha)})$. For every $x\in A$, $\supp(x)$ is at most countable, therefore, since $\mathscr A$ is uncountable,
    \begin{equation*}
        \left\vert\bigcup_{x\in A}\supp(x)\right\vert\le|A|\cdot\aleph_0<|\mathscr A|.
    \end{equation*}
    Find some $\beta\in\mathscr A\setminus\bigcup_{x\in A}\supp(x)$ and notice that $\|x+ e_\beta\|=1$ holds for every $x\in A$.
\end{proof}

\subsection{Tensor products}

In this subsection we give examples of projective and injective tensor products of Banach spaces which are transfinite (A)SQ. Our motivation for this are the known stability results of regular ASQ by taking tensor products coming from \cite{LLR2017, RZ2020}.

Let us begin with the projective tensor product, of which we briefly recall some notions. Recall that, given two Banach spaces $X$ and $Y$, the \textit{projective tensor product} of $X$ and $Y$, denoted by $X\pten Y$, is the completion of $X\otimes Y$ under the norm given by
\begin{equation*}
   \Vert u \Vert :=
   \inf\left\{
      \sum_{i=1}^n  \Vert x_i\Vert\Vert y_i\Vert
      : u=\sum_{i=1}^n x_i\otimes y_i
      \right\}.
\end{equation*}
It is known that $B_{X\pten Y}=\overline{\co}(B_X\otimes B_Y)
=\overline{\co}(S_X\otimes S_Y)$ \cite[Proposition~2.2]{Ryan2001}.
Moreover, given two Banach spaces $X$ and $Y$, it is well known that
$(X\pten Y)^*=L(X,Y^*)$ (see \cite{Ryan2001} for background on tensor products).

In \cite[Theorem~2.1]{RZ2020}, it was proved that, if $X$ and $Y$ are ASQ, then $X\pten Y$ is ASQ. The proof is based on averaging techniques in Banach spaces. In the following result we will obtain a transfinite version, which will give us more examples of transfinite (A)SQ spaces.

\begin{theorem}\label{theo:largeASQprojectensor}
    Let $\kappa$ be an uncountable cardinal. If $X$ and $Y$ are $\AandSQ{\kappa}$, then $X\pten Y$ is $\AandSQ{\kappa}$.
\end{theorem}
\begin{proof}

    Let us prove the ASQ case only, the other is similar. To this end, let $A\in\mathcal{P}_{<\kappa}(S_{X\pten Y})$ and $\varepsilon>0$. Since $S_{X\pten Y}=\cconv(S_X\otimes S_Y)$, for every $u\in A$ and $n\in\mathbb N$, we can find $m_n\in\mathbb N$, $\lambda^{u,n}_i\ge0$, $x^{u,n}_i\in S_X$ and $y^{u,n}_i\in S_Y$ for $i\in\{1,\ldots,m_n\}$ such that
    \begin{equation*}
        \left\|u-\sum_{i=1}^{m_n}\lambda^{u,n}_i x^{u,n}_i\otimes y^{u,n}_i\right\|\le1/n\text{ and }\sum_{i=1}^{m_n}\lambda^{u,n}_i=1.
    \end{equation*}
    Since $\kappa$ is uncountable, the sets $\{x^{u,n}_i:u\in A,n\in\mathbb N\text{ and }i\in\{1,\ldots,m_n\}\}$ and $\{y^{u,n}_i:u\in A,n\in\mathbb N\text{ and }i\in\{1,\ldots,m_n\}\}$ have cardinality $<\kappa$, therefore we can find $x\in S_X$ and $y\in S_Y$ satisfying
    \begin{equation*}
        \|x^{u,n}_i+x\|\le(1+\varepsilon)^{1/2}\text{ and }\|y^{u,n}_i+y\|\le(1+\varepsilon)^{1/2}
    \end{equation*}
    for all $u\in A,n\in\mathbb N$  and $i\in\{1,\ldots,m_n\}$.
    Thanks to \cite[Lemma~2.2]{RZ2020}, $\|x^{u,n}_i\otimes y^{u,n}_i+x\otimes y\|\le1+\varepsilon$ holds for every $u\in A$, $n\in\mathbb N$ and $i\in\{1,\ldots,m_n\}$. It is clear that
   \[\begin{split}
        \|u+x\otimes y\| & \le\left\|\sum_{i=1}^{m_n}\lambda^{u,n}_i ( x^{u,n}_i\otimes y^{u,n}_i+x\otimes y)\right\|+1/n\\ &  \le \sum_{i=1}^{m_n}\lambda^{u,n}_i \| x^{u,n}_i\otimes y^{u,n}_i+x\otimes y\|+1/n\le\\
        & \le(1+\varepsilon)\sum_{i=1}^{m_n}\lambda^{u,n}_i+1/n=1+\varepsilon+1/n
    \end{split}\]
    holds for every $u\in A$ and $n\in\mathbb N$. In other words, $\|x\otimes y+ u\|\le1+\varepsilon$ holds for every $u\in A$, and the proof is finished.
\end{proof}

\begin{remark}
In general, we cannot obtain that a projective tensor product $X\pten Y$ is $\ASQ{\kappa}$ if we only require one factor to be $\ASQ{\kappa}$. Indeed, if we take $X=c_0(\kappa)$ and $Y=\ell_p$ for $2<p<\infty$, we get, from \cite[Theorem 3.8]{LLR20172}, that $X\pten Y$ fails to be ASQ (it even contains a convex combination of slices of diameter smaller than 2).
\end{remark}

Now we turn our attention to when a space of operators can be transfinite ASQ, a study that will cover the injective tensor product too. Let $X$ and $Y$ be Banach spaces. Given an infinite cardinal $\kappa$, denote by
\begin{equation*}
    L_\kappa(Y,X):=\{T\in L(Y,X):\dens(T(Y))\le\kappa\}.
\end{equation*}
Using the ideas in \cite[Theorem~2.6]{LLR2017}, we get the following.

\begin{theorem}\label{theo:transfiASQinjective}
    Let $\lambda<\kappa$ be infinite cardinals, $X$ and $Y$ be non-trivial Banach spaces. Suppose that $X$ is $\AandSQ{\kappa}$.
    \begin{itemize}
        \item[(a)] If $H\subset L_\lambda(Y,X)$ is a closed subspace such that $Y^*\otimes X\subset H$, then $H$ is $\AandSQ{\kappa}$.
        \item[(b)] If $H\subset L_\lambda(Y^*,X)$ is a closed subspace such that $Y\otimes X\subset H$, then $H$ is $\AandSQ{\kappa}$.
    \end{itemize}
\end{theorem}
\begin{proof}
    Let us prove only (a) in the ASQ case. Fix $\mathscr T\in\mathcal{P}_{<\kappa}(S_H)$ and $\varepsilon>0$. Consider the subspace
    \begin{equation*}
        \mathscr{T}(Y):=\bigcup_{T\in\mathscr T}T(Y)
    \end{equation*}
    and notice that $\dens(\mathscr T(Y))\le|\mathscr T|\cdot\lambda<\kappa$. By assumption, there exists $x\in S_X$ satisfying
    \begin{equation*}
        \|z+ r x\|\le(1+\varepsilon)(\|z\|\vee|r|)\text{ for every }z\in\mathscr T(Y)\text{ and }r\in\mathbb R.
    \end{equation*}
    Fix any $y^*\in S_{Y^*}$, then the element $y^*\otimes x\in S_H$ satisfies, for every $T\in\mathscr T$ and $y\in S_Y$,
    \begin{equation*}
        \|(T+y^*\otimes x)(y)\|=\|T(y)+y^*(y)\cdot x\|\le(1+\varepsilon)(\|T(y)\|\vee |y^*(y)|)\le 1+\varepsilon.
    \end{equation*}
    If we pass to the $\sup$ on the left hand-side, we conclude that $\|T+y^*\otimes x\|\le1+\varepsilon$ holds for every $T\in\mathscr T$, as desired.
\end{proof}
Recall that, given two Banach spaces $X$ and $Y$, the \textit{injective tensor product of $X$ and $Y$}, denoted by $X\iten Y$, is the closure of the space of finite-rank operators from $Y^*$ to $X$. Taking this into account, the following corollary is clear from Theorem \ref{theo:transfiASQinjective}.

\begin{corollary}\label{coro:injectivetensASQ}
    Let $\kappa$ be an uncountable cardinal, $X$ and $Y$ non-trivial Banach spaces. If $X$ is $\AandSQ{\kappa}$, then $X\iten Y$ is $\AandSQ{\kappa}$.
\end{corollary}

\subsection{Ultrapowers}

In this subsection we will provide examples of ultrapowers of Banach spaces which are ASQ. Our motivation comes from \cite{Hardtke2018} where, in our language, it is proved that the ultrapower of a Banach space $X$ is $\SQ{\aleph_0}$ if, and only if, $X$ is ASQ.

Let us start with a bit of notation. Given a family of Banach spaces $\{X_\alpha:\alpha\in\mathscr A\}$, for an infinite set $\mathscr A$, we denote 
$$\ell_\infty(\mathscr A,X_\alpha):=\left\{f\colon\mathscr A\longrightarrow \prod\limits_{\alpha\in\mathscr A} X_\alpha: f(\alpha)\in X_\alpha\ \forall \alpha\text{ and }\bigvee_{\alpha\in\mathscr A}\Vert f(\alpha)\Vert<\infty\right\}.$$
Given a non-principal ultrafilter $\mathcal U$ over $\mathscr A$, consider $c_{0,\mathcal U}(\mathscr A,X_\alpha):=\{f\in \ell_\infty(\mathscr A,X_\alpha): \lim_\mathcal U \Vert f(\alpha)\Vert=0\}$. The \textit{ultrapower of $\{X_\alpha:\alpha\in\mathscr A\}$ with respect to $\mathcal U$} is
the Banach space
$$(X_\alpha)_\mathcal U:=\ell_\infty(\mathscr A,X_\alpha)/c_{0,\mathcal U}(\mathscr A,X_\alpha).$$
We will naturally identify a bounded function $f\colon\mathscr A\longrightarrow \prod\limits_{\alpha\in\mathscr A} X_\alpha$ with the element $(f(\alpha))_{\alpha\in\mathscr A}$. In this way, we denote by $(x_\alpha)_{\alpha,\mathcal U}$ or simply by $(x_\alpha)_\mathcal U$, if no confusion is possible, the coset in $(X_\alpha)_\mathcal U$ given by $(x_\alpha)_{\alpha\in\mathscr A}+c_{0,\mathcal U}(\mathscr A,(X_\alpha))$.

From the definition of the quotient norm, it is not difficult to prove that $\Vert (x_\alpha)_\mathcal U\Vert=\lim_\mathcal U \Vert x_\alpha\Vert$ holds for every $(x_\alpha)_\mathcal U\in (X_\alpha)_\mathcal U$.

Now we are ready to prove the following theorem.

\begin{theorem}\label{theo:stabilityultraproduct}
    Let $\mathscr A$ be an infinite set and $\{X_\alpha:\alpha\in\mathscr A\}$ a family of $\ASQ{\kappa}$ spaces. If $\mathcal U$ is a $\aleph_1$-incomplete non-principal ultrafilter over $\mathscr A$, then $(X_\alpha)_\mathcal U$ is $\SQ{\kappa}$.
\end{theorem}

\begin{proof}
    Since $\mathcal U$ is $\aleph_1$-incomplete we can find a function $f:\mathscr A\longrightarrow \mathbb R$ so that $f(\alpha)>0$ holds for every $\alpha\in\mathscr A$ and so that $\lim_\mathcal U f(\alpha)=0$.
    
    Let us now prove that $(X_\alpha)_\mathcal U$ is $\SQ{\kappa}$. To this end, fix a set $A\in\mathcal{P}_{<\kappa}(S_{(X_\alpha)_\mathcal U})$. Without loss of generality we can assume that $\Vert x(\alpha)\Vert=1$ holds for every $\alpha\in\mathscr A$ and every $x\in A$. Now, since $X_\alpha$ is $\ASQ{\kappa}$ we can find, for every $\alpha\in\mathscr A$, an element $y_\alpha\in S_{X_\alpha}$ so that
    $$\Vert x(\alpha)+ y_\alpha\Vert\leq 1+f(\alpha)$$
    holds for every $x\in A$.

Now consider $(y_\alpha)_{\mathcal U}\in S_{(X_\alpha)_\mathcal U}$, and let us prove that it satisfies the desired inequality. To this end fix $x\in A$ and notice that
$$\Vert x+ (y_\alpha)_\mathcal U\Vert=\lim_{\mathcal U}\Vert x(\alpha)+ y_\alpha\Vert\leq \lim_\mathcal U (1+f(\alpha))=1,$$
as requested.
\end{proof}

It is natural, in view of what happen with the behaviour of ASQ in ultrapowers, to ask whether $(X_\alpha)_\mathcal U$ ASQ implies $X_\alpha$ ASQ for some $\alpha\in\mathscr A$. The following example shows that the answer is no.

\begin{example}\label{exam:ultraln}
    Let $\kappa$ be an infinite cardinal and set $X_n:=\ell_n(\kappa)$, where $n\in\mathbb{N}_{\geq 2}$. Let $\mathcal U$ be a non-principal ultrafilter over $\mathbb N$ and consider $X:=(X_n)_\mathcal U$. Let us prove that $X$ is $\SQ{\kappa}$ in spite of $X_n$ being reflexive for every $n\in\mathbb{N}_{\geq 2}$. Fix $A\in\mathcal{P}_{<\kappa}(S_X)$ and assume that $x(n)$ has norm one for each $x\in A$ and $n\in\mathbb{N}_{\geq 2}$.
    
    By the same argument as in Theorem \ref{theo:kappaASQnotconc0} we get elements $y_n\in S_{X_n}$ so that $\Vert x(n)+ y_n\Vert\leq 2^\frac{1}{n}$ holds for every $x\in A$ and $n\in\mathbb{N}_{\geq 2}$. It is not difficult to prove, as before, that
    $$\Vert x+ (y_n)_\mathcal U\Vert=\lim_\mathcal U \Vert x(n)+ y_n\Vert\leq \lim_\mathcal U 2^\frac{1}{n}=1$$
    holds for every $x\in A$, and the proof is finished.
    \end{example}

\section{Connections with other properties}\label{section:connectotherprop}

It is known that almost square Banach spaces have deep connections with other properties of the geometry of Banach spaces such as diameter two properties, octahedrality and the intersection property (see \cite{ALL2016}). The aim of the present section is to derive similar connections with transfinite counterparts of the abovementioned properties.

We will be specially interested in the connection between transfinite versions of almost squareness and octahedrality, because, as a consequence of our work, we will solve an open question from \cite{CLL2021}. In order to do so, let  us start with the following definition from \cite{CLL2021}.

\begin{definition}(see \cite[Definitions~2.3 and 5.3]{CLL2021})\label{def: kappa-octahedrality}
	Let $X$ be a Banach space and $\kappa$ an uncountable cardinal.
	\begin{itemize}
	    \item[(a)] We say that $X$ is \emph{$<\kappa$-octahedral} if, for every subspace $Y\subset X$ with $\dens(Y)<\kappa$ and $\varepsilon>0$, there exists $x\in S_X$ such that for all $r\in\mathbb{R}$ and $y\in Y$ we have $\|y+r x\|\ge(1-\varepsilon)(\|y\|+|r|)$.
	    \item[(b)] We say that $X$ fails the \textit{$(-1)$-BCP$_{<\kappa}$} if , for every subspace $Y\subset X$ with $\dens(Y)<\kappa$, there exists $x\in S_X$ such that for all $r\in\mathbb{R}$ and $y\in Y$ we have $\|y+r x\|=\|y\|+|r|$.
	\end{itemize}
	
	If $\kappa=\aleph_0$, analogous properties are defined by considering finite-dimensional subspaces $Y\subset X$ instead.
\end{definition}

It is known that if a Banach space $X$ is ASQ, then $X^*$ is octahedral \cite[Proposition 2.5]{ALL2016}. In order to solve the abovementioned question, our aim will be to establish a transfinite version of this result. We will perform this proof, however, from a more general principle using transfinite versions of the strong diameter two property.

\begin{definition}
    (see \cite[Definitions~2.11 and 2.12]{CLL2022}) Let $X$ be a Banach space and $\kappa$ an infinite cardinal.
    \begin{itemize}
        \item[(a)] We say that $X$ has the \textit{SD2P$_{<\kappa}$} if, for every $A\in\mathcal P_{<\kappa}(S_{X^*})$ and $\varepsilon>0$, there exist $B\subset S_X$ and $x^*\in S_{X^*}$ such that $x^*(x)\ge1-\varepsilon$ holds for all $x\in B$ and $B$ $(1-\varepsilon)$-norms $A$, that is
    \begin{equation*}
        \bigvee_{x\in B}y^*(x)\ge1-\varepsilon\text{ holds for every }y^*\in A.
    \end{equation*}
    \item[(b)] We say that $X$ has the \textit{1-ASD2P$_{<\kappa}$} if, for every $A\in\mathcal P_{<\kappa}(S_{X^*})$, there exist $B\subset S_X$ and $x^*\in S_{X^*}$ such that $x^*(x)=1$ holds for all $x\in B$ and $B$ is norming for $A$, that is $B$ $1$-norms $A$.
    \end{itemize}
\end{definition}

Recall that, for all infinite cardinal $\kappa$, a Banach space $X$ has the SD2P$_{<\kappa}$ if, and only if, $X^*$ is $<\kappa$-octahedral and that if $X$ has the 1-ASD2P$_{<\kappa}$, then $X^*$ fails the $(-1)$-BCP$_{<\kappa}$ \cite[Theorem~3.2 and Proposition~3.6]{CLL2022}.

\begin{proposition}\label{prop:connecASQSd2P}
    Let $X$ be a Banach space and $\kappa$ an infinite cardinal. If $X$ is $\ASQ{\kappa}$, then $X$ has the SD2P$_{<\kappa}$. Moreover, if $\kappa$ is uncountable and $X$ is $\SQ{\kappa}$, then $X$ has the 1-ASD2P$_{<\kappa}$.
\end{proposition}
\begin{proof}
    We begin the proof by noticing that, if $X$ is $\ASQ{\kappa}$, then it clearly satisfies a transfinite version of the \textit{symmetric strong diameter two property} \cite{ANP2019,HLLN2019}. Namely, for every $A\in\mathcal{P}_{<\kappa}(S_{X^*})$ and $\varepsilon>0$ there exist $B\subset S_X$ and $y\in S_X$ such that $B$ $(1-\varepsilon)$-norms $A$ and $y\pm B\subset(1+\varepsilon)B_X$.
    
    From this property we now deduce that $X$ has the SD2P$_{<\kappa}$. For this purpose, fix $A\in\mathcal{P}_{<\kappa}(S_{X^*})$ and $\varepsilon>0$. Find $B\subset S_X$ and $y\in S_X$ such that $B$ $(1-\varepsilon/3)$-norms $A$ and $y\pm B\subset (1+\varepsilon/3)B_X$.
    
    We claim that $y+B$ also $(1-\varepsilon)$-norms $A$. In fact, for every $x^*\in A$ we can find $x\in B$ such that $x^*(x)\ge1-\varepsilon/3$ and therefore
    \begin{equation*}
        1=\|x^*\|\ge\frac{x^*(x\pm y)}{1+\varepsilon/3}\ge\frac{1-\varepsilon/3\pm x^*(y)}{1+\varepsilon/3},
    \end{equation*}
    hence $|x^*(y)|\le2\varepsilon/3$. We conclude that $x^*(x+y)\ge1-\varepsilon$, that is the claim.
    
    In order to conclude, we need to find $x^*\in S_{X^*}$ such that $x^*(x+y)\ge1-\varepsilon$ holds for every $x\in A$. Any $x^*\in S_{X^*}$ that attains its norm at $y$ satisfies the desired condition, in fact, for every $x\in A$, we have that
    \begin{equation*}
        1=\|x^*\|\ge\frac{x^*(y\pm x)}{1+\varepsilon/3}=\frac{1\pm x^*(x)}{1+\varepsilon/3},
    \end{equation*}
    hence $|x^*(x)|\le\varepsilon/3$ and therefore $x^*(x+y)\ge1-\varepsilon/3$.
    
    The additional part follows by repeating the same proof with $\varepsilon=0$ and taking into account that, given any element $x^*\in S_{X^*}$, then the set $\{x^*\}$ can be normed using a countable set.
\end{proof}

\begin{remark}\label{rem:solutionCCL2021}
In \cite[Question 1]{CLL2021} it was asked whether every $<\kappa$-octahedral Banach space must contain an isomorphic copy of $\ell_1(\kappa)$. Observe that, by Proposition \ref{prop:connecASQSd2P}, the dual of the space exhibited in Example \ref{theo:kappaASQnotconc0} provides a negative answer for every infinite cardinal $\kappa$.
\end{remark}

\section{Parametric ASQ spaces}\label{section:generalisations}

In this last section we study a further generalization of ASQ spaces.

\begin{definition}\label{def: parametric ASQ}
 Let $X$ be a Banach space, $\kappa$ be a cardinal, and $r,s\in(0,1]$. We say that $X$ is \textit{$\rsSQ{r}{s}{\kappa}$} if, for every set $A\in\mathcal{P}_{<\kappa}(S_X)$, there exists $y\in S_X$ satisfying
 \begin{equation*}
     \|rx \pm sy\|\le1\text{ for every }x\in A.
 \end{equation*}
 We say that $X$ is \textit{$\rsSQ{<r}{s}{\kappa}$} if it is $\rsSQ{t}{s}{\kappa}$ for all $t\in(0,r)$ and similar meaning is given to being \textit{$\rsSQ{r}{<s}{\kappa}$}.
 As before, we put ``$\kappa$'' instead of ``$< \kappa$'' in the definitions above to mean the non-strict inequality on the cardinals.
\end{definition}

It is clear that being $\ASQ{\kappa}$ coincides with being $\rsSQ{<1}{<1}{\kappa}$ and that being $\SQ{\kappa}$ corresponds to being $\rsSQ{1}{1}{\kappa}$. 

\begin{remark}
Let $s\in (0,1]$. The quantitative version of almost squareness studied in \cite{OSZ2020}, named $s$-ASQ, corresponds to the space being $\rsSQ{<1}{<s}{\aleph_0}$.
\end{remark}

Before proceeding further, let us prove that every $\rsSQ{r}{s}{\kappa}$ space can be described through subspaces of density character $<\kappa$. 

\begin{lemma}\label{lem: (r,s)-SQ}
    Let $X$ be a Banach space, $x,y\in S_X$ and $r,s\in(0,1]$. If $\|rx+ sy\|\le1$, then $\|r'x+ s'y\|\le1$ holds for every $r'\in(0,r]$ and $s'\in(0,s]$.
\end{lemma}
\begin{proof}
    Suppose without loss of generality that $s/s'\le r/r'$ and notice at first that
    \begin{align*}
        \|r'x+ sy\| & =\frac{r'}{r}\left\lVert rx+ \frac{rs}{r'}y\right\rVert\le\frac{r'}{r}\left(\|rx+ sy\|+s\left(\frac{r}{r'}-1\right)\right)\le\\
        & \le\frac{r'}{r}+s\left(1-\frac{r'}{r}\right)\le1.
    \end{align*}
    We just proved that $\|r'x+ sy\|\le1$ holds for every $r'\in(0,r]$. We now can conclude since
    \begin{equation*}
        \|r'x+ s'y\|=\frac{s'}{s}\left\lVert\frac{r's}{s'}x+ sy\right\rVert\le\frac{s'}{s}\le1,
    \end{equation*}
    where we used the first part of the proof together with the fact that $r's/s'\le r$.
\end{proof}

Notice that, thanks to Lemma~\ref{lem: (r,s)-SQ}, we can already conclude that $\rsSQ{r}{s}{\kappa}$ implies $\rsSQ{r'}{s'}{\kappa}$ whenever $r'\le r$ and $s'\le s$. 

\begin{theorem}\label{thm: rs-sq for subspaces}
    Let $X$ be a Banach space, $\kappa$ an uncountable cardinal and $r,s\in(0,1]$. Then $X$ is $\rsSQ{r}{s}{\kappa}$ if, and only if, for every subspace $Y\subset X$ with $\dens(Y)<\kappa$, there exists $x\in S_X$ satisfying
    \begin{equation*}
        \|ry+stx\|\le\|y\|\vee|t|\text{ for all }y\in Y\text{ and }t\in\mathbb R.
    \end{equation*}
\end{theorem}
\begin{proof}
    One implication is obvious. For the vice-versa, notice that we only need to prove the claim when $t\ge0$. For this purpose, fix a subspace $Y\subset X$ with $\dens(Y)<\kappa$ and find $x\in S_X$ such that
    \begin{equation*}
        \|ry+sx\|\le1\text{ holds for every }y\in S_Y.
    \end{equation*}
    Such a $x$ exists by a density argument. Fix $t\ge0$, $y\in Y$ and notice that, thanks to Lemma~\ref{lem: (r,s)-SQ},
    \begin{align*}
        \|ry+stx\| & =(\|y\|\vee t)\left\|\frac{r\|y\|}{\|y\|\vee t}\frac{y}{\|y\|}+\frac{st}{\|y\|\vee t}x\right\|\le(\|y\|\vee t).
    \end{align*}
\end{proof}

Now, let us point out some geometrical considerations.

\begin{lemma}\label{lem: SQ does not have extreme points nor is strictly convex}
Let $X$ be a Banach space.
\begin{enumerate}
    \item[(a)] If $X$ is $\rsSQ{1}{<1}{\aleph_0}$  (or just $\rsSQeq{1}{k}{1}$ for some $k \in (0,1]$), then $B_X$ cannot contain any extreme point.
    \item[(b)] If $X$ is $\rsSQ{<1}{1}{\aleph_0}$ (or just $\rsSQeq{k}{1}{1}$ for some $k \in (0,1]$), then $X$ is not strictly convex.
\end{enumerate}
\end{lemma}
\begin{proof}
(a). Let $x\in S_X$. By our assumption we can find $y\in B_X\setminus\{0\}$ such that $\|x\pm y\|\le1$. Notice that
\[
x=\frac{x+y}{2}+\frac{x-y}{2}.
\]
Thus $x$ is a middle point of two distinct elements of $B_X$ and cannot be an extreme point.

(b). Let $x\in kB_X\setminus\{0\}$. By our assumption we can find $y\in S_X$ such that $\|x\pm y\|\le1$. Observe that
\begin{equation*}
    1=\|y\|=\left\lVert\frac{y+x}{2}+\frac{y-x}{2}\right\rVert.
\end{equation*}
From this we deduce, through a simple contradiction argument, that $\|x\pm y\|=1$ and that $y$ is a norm-1 element which is a middle point of two distinct norm-1 elements, thus proving the claim.
\end{proof}

Let us state a simple but useful observation about how $\rsSQ{r}{s}{\kappa}$ properties pass from a component to the $\infty$-sum.

\begin{proposition}\label{prop: (r,s)-SQ under ell_infty sum}
Let $X$ and $Y$ be non-trivial Banach spaces, $r, s \in (0, 1]$, and let $\kappa$ be a cardinal. If $X$ is $\rsSQ{r}{s}{\kappa}$, then $X \oplus_\infty Y$ is $\rsSQ{r}{s}{\kappa}$.
\end{proposition}
\begin{proof}
    Fix a set $\{x_\gamma \oplus_\infty y_\gamma\}_{\gamma \in \Gamma} \subset S_{X \oplus_\infty Y}$ with $|\Gamma| < \kappa$. Find $z \in S_X$ such that
  $\|\frac{rx_\gamma}{\|x_\gamma\|} \pm sz\| \leq 1$ for all $x_\gamma \neq 0$. By Lemma~\ref{lem: (r,s)-SQ} we have $\|rx_\gamma + sz\| \leq 1$ (even when $x_\gamma = 0$), so that $\|r(x_\gamma \oplus_\infty y_\gamma) + s(z \oplus_\infty 0) \| \leq 1$ for all $\gamma \in \Gamma$.
\end{proof}

Now let us show some first easy examples.

\begin{example}
    It is easy to check that $c_0$ is $\rsSQ{1}{<1}{\aleph_0}$. In fact, for every $x_1,\ldots,x_n\in S_{c_0}$ and $\varepsilon>0$ we can find $m\in\mathbb N$ such that $|x_i(m)|\le\varepsilon$ holds for $i\in\{1,\ldots,n\}$ and it is clear that $\|x_i+(1-\varepsilon)e_m\|\le1$.
    
    Even more is true, given any sequence $\{X_n:n\in\mathbb N\}$ of Banach spaces, $c_0(\mathbb N,X_n)$ is $\rsSQ{1}{<1}{\aleph_0}$. On the other hand it is trivial to verify that $c_0$ is not $\SQ{\aleph_0}$ by considering the element $x=\sum_{n=1}^\infty n^{-1}e_n\in S_{c_0}$.
\end{example}

Following the same ideas as in \cite[Theorem~2.5]{ANP2019}, the previous argument can be exploited also for proving more in general that \textit{somewhat regular subspaces} of $C_0(X)$ spaces, where $X$ is some non-compact locally compact and Hausdorff space, are $\rsSQ{1}{<1}{\aleph_0}$. 

With similar ideas we can slightly improve the renorming result stated in \cite[Theorem~2.3]{BLR2016}.

\begin{theorem}\label{thm: (r,s)-SQ renorming}
    A Banach space $X$ contains an isomorphic copy of $c_0$ if and only if it admits an equivalent $\rsSQ{1}{<1}{\aleph_0}$ norm.
\end{theorem}
\begin{proof}
    Assume that $X$ contains a subspace isometric to $c_0$. Then there is a subspace $Z$ of $\Xss$ so that $\Xss=\ell_\infty \oplus Z$. Consider the norm $\nn\cdot$ on $\ell_\infty$ described in Example \ref{ex: ell_infty rigid ASQ}. Now, consider on $\Xss$ the equivalent norm $\nn\cdot'$ so that $(\Xss,\nn\cdot') =(\ell_\infty,\nn\cdot)\oplus_\infty Z$. By Corollary \ref{cor: infty sum}, $(\Xss,\nn\cdot')$ is $\SQ{\aleph_0}$ because $(\ell_\infty,\nn\cdot)$ is $\SQ{\aleph_0}$.
    
    Now let $x_1=(u_1,z_1),\dots,x_k=(u_k,z_k)\in S_X$ and $\varepsilon>0$. Keeping in mind the notation as in Example \ref{ex: ell_infty rigid ASQ}, find $n\in\mathbb{N}$ such that $1/n<\varepsilon$ and define $y:=(1-\varepsilon)e_m\in B_{c_0}$, where $m\in A_n$. Then a similar calculation as in the proof of \cite[Theorem~2.3]{BLR2016} yields that the element $(y,0)\in (1-\varepsilon)B_X$ and that $\|x_i+ (y,0)\|\leq 1$ holds for every $i\in\{1,\dots,k\}$. Hence, $X$ is $\rsSQ{1}{<1}{\aleph_0}$.
    
    For the converse recall that every Banach space with $\rsSQ{1}{<1}{\aleph_0}$ norm is ASQ and every ASQ space is known to contain $c_0$ by \cite{ALL2016}.
\end{proof}

Notice that, if $\kappa$ is an infinite cardinal and a Banach space $X$ is $\rsSQ{1}{<1}{\kappa}$, then a simple transfinite induction proves that $X$ contains an isomorphic copy of $c_0(\kappa)$.
Thus, in the case when $\kappa$ is uncountable, the condition $\rsSQ{1}{<1}{\kappa}$ is different from $\ASQ{\kappa}$ due to Example \ref{theo:kappaASQnotconc0}.

Before giving more examples, let us prove a variation of Theorem~\ref{thm: ell_infty sum} that we will need later.

\begin{theorem}\label{thm: (r,s)-SQ under ell_infty sum}
    Let $\{X_\alpha:\alpha\in\mathscr A\}$ be a family of Banach spaces and $\kappa$ an infinite cardinal. If for every $r\in(0,1)$ there are infinitely many $\alpha\in\mathscr A$ such that $X_\alpha$ is $\rsSQ{r}{r}{\kappa}$, then $\ell_\infty(X_\alpha)$ is $\rsSQ{<1}{1}{\kappa}$.
\end{theorem}
\begin{proof}
    Fix $r\in(0,1)$ and $A\in\mathcal{P}_{<\kappa}(S_{\ell_\infty(X_\alpha)})$. For every $s\in(r,1)$ we can find $\alpha(s)\in\mathscr A$ and $y_s\in S_{X_{\alpha(s)}}$ satisfying
    \begin{equation*}
        \|sx(\alpha(s))+ sy_s\|\le1\text{ for all }x\in A.
    \end{equation*}
    By our assumption, we can assume that, if $s\not=s'$, then $\alpha(s)\not=\alpha(s')$. Define $y\in S_{\ell_\infty(X_\alpha)}$ by
    \begin{equation*}
        y(\alpha):=\bigg\{\begin{array}{ll}
            sy_s & \text{if }\alpha=\alpha(s)\text{ for some }s\in(r,1),\\
            0 & \text{otherwise.}
        \end{array}
    \end{equation*}
    Thanks to Lemma~\ref{lem: (r,s)-SQ}, we conclude that
    \begin{equation*}
        \|rx+ y\|_\infty=1\vee\bigvee_{s\in(r,1)}\|rx(\alpha(s))+ sy_s\|\}\le1
    \end{equation*}
    holds for every $x\in A$.
\end{proof}

Eventually we can present more examples of $\rsSQ{r}{s}{\kappa}$ spaces that will also show that these properties are actually distinct from the regular $\AandSQ{\kappa}$.

\begin{example}
    There exists a Banach space $X$ which is $M$-embedded and strictly convex \cite[P. 168]{hww}. Therefore $X$ is ASQ \cite[Corollary~4.3]{ALL2016}, however it is not $\rsSQ{1}{<1}{\aleph_0}$ nor $\rsSQ{<1}{1}{\aleph_0}$ by Lemma~\ref{lem: SQ does not have extreme points nor is strictly convex}.
\end{example}

\begin{example}
    In the proof of Example \ref{theo:kappaASQnotconc0}, it is shown that the Banach space $\ell_n(\kappa)$ is $\rsSQ{2^{-1/n}}{2^{-1/n}}{\kappa}$. Thus, $X:=\ell_\infty(\ell_n(\kappa))$ is $\rsSQ{<1}{1}{\kappa}$, thanks to Theorem~\ref{thm: (r,s)-SQ under ell_infty sum}, but it is not $\rsSQ{1}{<1}{\aleph_0}$, by Lemma~\ref{lem: SQ does not have extreme points nor is strictly convex}, since it is a dual space.
\end{example}

\section*{Acknowledgment}
The research of A. Avil\'es and A. Rueda Zoca was supported by MTM2017-86182-P (funded by MCIN/AEI/10.13039/501100011033 and ``ERDF A way of making Europe'') and by Fundaci\'{o}n S\'{e}neca [20797/PI/18].

The research of S. Ciaci, J. Langemets and A. Lissitsin was supported by the Estonian Research Council grants (PSG487) and (PRG1901).

The research of A. Rueda Zoca was also supported by MCIN/AEI/10.13039/\\ 501100011033 (Spain) Grant PGC2018-093794-B-I00, by Junta de Andaluc\'ia Grant A-FQM-484-UGR18 and by Junta de Andaluc\'ia Grant FQM-0185.

\end{document}